\documentclass[reqno,a4paper,12pt]{amsart}

\numberwithin{equation}{section}
\usepackage{amsmath, amsthm, amssymb, amscd, accents, amsfonts}
\usepackage{mathtools}
\usepackage{url}
\usepackage{mathrsfs,dsfont}
\usepackage{datetime}
\usepackage{hyperref}
\usepackage{xcolor}

\usepackage[no-math]{fontspec}
\usepackage[frozencache=true,cachedir=minted-cache]{minted}

\newmintinline[lean]{lean4}{bgcolor=white}
\newminted[leancode]{lean4}{escapeinside=!!, breaklines, fontsize=\normalsize}
\usepackage[sort,nocompress]{cite}
\mathtoolsset{showonlyrefs}

\newtheorem{definition}{Definition}[section]
\newtheorem{theorem}[definition]{Theorem}
\newtheorem*{theorem*}{Theorem}
\newtheorem{lemma}[definition]{Lemma}

\newtheorem{proposition}[definition]{Proposition}

\newtheorem{example}[definition]{Example}

\def\N{{\mathbb N}}
\def\Z{{\mathbb Z}}
\def\R{{\mathbb R}}
\def\T{{\mathbb T}}
\def\C{{\mathbb C}}
\def\Q{{\mathbb Q}}

\newcommand{\Rd}{{\R^d}}

\newcommand{\ltd}{{L^2(\R^d)}}

\allowdisplaybreaks

\begin{document}

\title[On the existence problem of regular Gabor frames]{On the existence problem of regular \\ Gabor frames}

\author[Jaume de Dios Pont]{Jaume de Dios Pont}
\address{Center for Data Science, New York University, New York, New York 10011, USA}
\email{jdedios@nyu.edu}

\author[Lukas Liehr]{Lukas Liehr}
\address{Department of Mathematics, Bar-Ilan University, Ramat-Gan 5290002, Israel}
\email{lukas.liehr@biu.ac.il}

\author[Mitchell A. Taylor]{Mitchell A. Taylor}
\address{ Department of Mathematics, ETH Z\"urich, Ramistrasse 101, 8092 Z\"urich, Switzerland}
\email{mitchell.taylor@math.ethz.ch}

\date{\today}

\begin{abstract}
For every dimension $d > 1$, we establish explicit criteria on lattices $\Lambda \subset \mathbb{R}^{2d}$ with density $D(\Lambda) > 1$ such that no function with a continuous Zak transform generates a Gabor frame along $\Lambda$. In particular, this gives a negative answer to the existence problem of Gabor frames with window functions in the Schwartz space, the Feichtinger algebra, and the Fourier-invariant Wiener space.
Our result is based on a characterization of when a collection of quasiperiodic functions admits a common zero, which may be of independent interest. We also include a formalization of our main result in Lean 4.
\end{abstract}

\maketitle

\section{Introduction and main results}

\subsection{}
For $g \in L^2(\mathbb{R}^d)$ and $\lambda=(x,\omega)\in\R^{2d}$, let $g_\lambda(t) := e^{2\pi i \omega\cdot t}g(t-x)$ denote the time-frequency shift of $g$ with respect to $\lambda$. The Gabor system of $g$ along a set $\Lambda \subseteq \R^{2d}$ is the collection of time-frequency shifts of $g$ with respect to the elements of $\Lambda$, i.e., 
$$
\mathbf{G}(g,\Lambda):=\{g_\lambda:\lambda\in\Lambda\}.
$$
The function $g$ is commonly referred to as the window function of $\mathbf{G}(g,\Lambda)$.
A fundamental problem in time-frequency analysis is to determine for which $g$ and $\Lambda$ the system $\mathbf{G}(g,\Lambda)$ forms a frame of $L^2(\mathbb{R}^d)$; that is, when there exists constants $A,B>0$ such that for all $f \in L^2(\Rd)$ one has the inequalities
$$
A\|f\|^2 \leq \sum_{\lambda\in\Lambda} |\langle f,g_\lambda\rangle|^2 \leq B\|f\|^2.
$$
This question has been studied extensively under the structural assumption that $\Lambda$ is a lattice, i.e., $\Lambda=M\mathbb{Z}^{2d} = \{ Mz : z \in \Z^{2d} \}$
for some invertible matrix $M\in\mathbb{R}^{2d\times 2d}$ \cite{Groechenig, belov2023gabor,grochenig2023totally,GroechenigRomeroStoeckler,1111}.

\subsection{}
In what follows, let $\Lambda \subseteq \R^{2d}$ be a lattice and let $D(\Lambda)$ be its density. It is well-known that if $\mathbf{G}(g,\Lambda)$ is a frame, then $\Lambda$ satisfies the necessary density condition $D(\Lambda)\geq 1$ \cite{heil2007history}.
A theorem of Bekka \cite{bekka2004square}, which is central to the study of Gabor frames, shows that the condition $D(\Lambda)\geq 1$ is also sufficient in an existential sense: whenever a lattice satisfies $D(\Lambda) \geq 1$, then there exists some $g \in \ltd$ for which $\mathbf{G}(g,\Lambda)$ is a frame for $L^2(\mathbb{R}^d)$. However, Bekka's argument is non-constructive and it leaves open whether one may choose $g$ with prescribed regularity or decay, such as belonging to the Schwartz space $\mathcal{S}(\Rd)$. While the existence of a Gabor frame with window in $\mathcal{S}(\Rd)$ for lattices with critical density $D(\Lambda)=1$ is excluded by the Balian-Low theorem, the existence for lattices with supercritical density $D(\Lambda)>1$ constitutes a well-known open problem in dimensions $d>1$ (the existence of a Schwartz window in the case $d=1$ was settled affirmatively by Seip-Wallstén and Lyubarskii \cite{Seip1992,Seip+1992+91+106,lyubarskiiframes}). Precisely, it is conjectured that for every lattice $\Lambda$ with $D(\Lambda)>1$ there exists $g \in \mathcal{S}(\Rd)$ such that $\mathbf{G}(g,\Lambda)$ is a frame \cite[p.~225]{grochenig2011multivariate}.
Another formulation of the conjecture asks if one can choose $g$ as an element of the Feichtinger algebra $S_0(\Rd)$, which is the collection of all functions $g \in \ltd$ such that
$$
\int_{\R^{2d}} |\langle g,g_z \rangle| \, dz < \infty.
$$
This version of the existence problem of regular Gabor frames is recorded, for instance, in \cite{jakobsen2020duality}. In fact, it is equivalent to the existence problem of Schwartz Gabor frames \cite[Proposition 4.4]{grochenig2020balian}.

\subsection{}
In the present paper we consider the class of functions $C_Z(\Rd)$ which consists of all $f \in \ltd$ such that $f$ has a continuous Zak transform, i.e.,
$$
Zf(u,v):=\sum_{n\in\mathbb Z^d} f(u-n)e^{2\pi i\,n\cdot v}
$$
is continuous on $\R^d \times \R^d$.
The space $C_Z(\Rd)$ is an enlargement of the Feichtinger algebra and contains the so-called Fourier invariant Wiener space $W_0(\Rd)$, whose definition goes back to Wiener's seminal work on Tauberian theorems \cite{wiener1932tauberian}. A function $f$ belongs to $W_0(\Rd)$ if
$$
\sum_{n \in \Z^d} \| f(\cdot + n) \|_{L^\infty([0,1]^d)} < \infty, \quad \sum_{n \in \Z^d} \| \hat f(\cdot + n) \|_{L^\infty([0,1]^d)} < \infty,
$$
where $\hat f$ denotes the Fourier transform of $f$. Specifically, the aforementioned function spaces are related by the strict inclusions
\begin{equation}\label{eq:space_inclusions}
    \mathcal{S}(\Rd) \subsetneq S_0(\Rd) \subsetneq W_0(\Rd) \subsetneq C_Z(\Rd).
\end{equation}
The fact that $S_0(\Rd)$ is a proper subspace of $W_0(\Rd)$ follows from a theorem of Losert \cite[Theorem 2]{losert1980characterization}, who proved that when defining analogous spaces $S_0(G)$ and $W_0(G)$ on a locally compact abelian group $G$, the equality $S_0(G) = W_0(G)$ occurs if and only if $G$ is either compact or discrete. The remaining strict inclusions are discussed in Section \ref{sec:wiener}.

\subsection{}
In contrast to what was expected from the literature and the series of positive results discussed below, we prove in every dimension $d > 1$ that there is a general obstruction to the frame property of Gabor systems $\mathbf{G}(g,\Lambda)$. Our obstruction holds for window functions $g$ belonging to the largest space in \eqref{eq:space_inclusions}, i.e, $g \in C_Z(\mathbb{R}^d)$. As a consequence, this yields a negative resolution to the existence problem for Gabor frames in $\mathcal{S}(\Rd)$ and $S_0(\Rd)$.

\begin{theorem}\label{thm:conj}
    For every $d > 1$ there exists a lattice $\Lambda \subset \R^{2d}$ with density $D(\Lambda) > 1$ having the following property: there exists no $g \in C_Z(\mathbb{R}^d)$ such that $\mathbf G(g,\Lambda)$ is a frame for $L^2(\mathbb{R}^d)$.
\end{theorem}

In combination with the theorems of Seip–Wallstén and Lyubarskii, we obtain that the property that every lattice in $\mathbb{R}^{2d}$ with supercritical density admits a Gabor frame with window in $C_Z(\mathbb{R}^d)$ is equivalent to the condition $d = 1$. This indicates a significant difference in the behavior of the existence problem for Gabor frames in higher dimensions compared to the univariate case.

\subsection{}\label{g}
Our theorem complements a series of positive results on the existence problem of regular Gabor frames. In dimension $d=1$, Seip-Wallstén and independently Lyubarskii proved that the Gaussian
Gabor system $\mathbf{G}(e^{-\pi t^2},\Lambda)$ is a frame for every lattice $\Lambda \subseteq \mathbb{R}^2$ satisfying $D(\Lambda)>1$ \cite{Seip1992,Seip+1992+91+106,lyubarskiiframes}. On the other hand, in dimensions $d > 1$, Gröchenig and Lyubarskii constructed lattices $\Lambda \subseteq \R^{2d}$ with $D(\Lambda) > 1$ for which the Gaussian Gabor system $\mathbf{G}(e^{-\pi |t|^2},\Lambda)$ fails to be a frame \cite{grochenig2020sampling}. This implies that the Gaussian cannot serve as a universal Schwartz window in all dimensions. The study of Gaussian Gabor frames in dimensions $d > 1$ remains an active research field. See, for example, the recent contributions by Luef and Wang \cite{luef2023gaussian} and by Romero, Ulanovskii and Zlotnikov \cite{romero2024sampling}. 

Pfander, Rashkov, and Wang demonstrated that in dimension $d = 2$ and for separable lattices that satisfy specific geometrical conditions, one can construct smooth functions $ g $ with compact support such that $\mathbf{G}(g,\Lambda)$ is a frame \cite{pfander2012geometric}. Their approach is based on an existence result of fundamental domains of pairs of lattices \cite{han2001lattice} (see also the recent contribution \cite{grepstad2026bounded} on bounded fundamental domains). 
B{\'e}dos, Enstad, and van Velthoven showed in \cite{bedos2022smooth} that for every $d 
\geq 1$ and every lattice $\Lambda \subseteq \R^{2d}$ satisfying Kleppner's condition \cite{Kleppner1962}, there exists a Schwartz function $g$ such that $\mathbf{G}(g,\Lambda)$ is a frame. Jakobsen and Luef proved the existence of a Schwartz function $g$ generating a Gabor frame in the case of non-rational lattices \cite{jakobsen2020duality}.
Enstad, Thiel, and Vilalta obtained a partial positive result for rational lattices \cite{enstad2025criteria} (see also \cite{enstad2026zstability}) by introducing an arithmetic invariant $n_\Lambda$ and proving that a rational lattice $\Lambda$ admits a Schwartz Gabor frame whenever
\begin{equation}\label{eq:nLambda_inequality}
    \frac{1}{D(\Lambda)} <1-\frac{d-1}{n_\Lambda}.
\end{equation}
The article \cite{liu2026full} also addresses the existence problem for certain irregular sets which are far from being lattices.
\subsection{}
Given a lattice $L \subset \R^d$ we denote by $L^*$ its dual lattice. Our next result gives an algebraic obstruction for the existence of a regular Gabor frame. In fact, our obstruction holds in the general setting of multi-window Gabor systems, which are systems of the form $\bigcup_{\ell=1}^r \mathbf G(g_\ell,\Lambda)$ for some $g_1, \dots, g_r \in \ltd$.

\begin{theorem}\label{thm:main}
Let $d > 1$ and let $L=A\Z^d$ be a lattice with $A \in \mathrm{GL}_d(\Q)$. If $\Lambda \subseteq \R^{2d}$ is a lattice containing $L \times L^*$ as a sublattice of index $N \in \N$, then the following are equivalent:
\begin{enumerate}
\item There exist $g_1,\dots,g_r\in C_Z(\mathbb R^d)$ such that
$
\bigcup_{\ell=1}^r \mathbf G(g_\ell,\Lambda)
$
is a frame for $L^2(\mathbb R^d)$,
\item $rN>d$.
\end{enumerate}
Moreover, if $rN>d$ then each $g_\ell$ may be chosen in such a way that $ g_\ell \in  \mathcal S(\mathbb R^d)$ and $\bigcup_{\ell=1}^r\mathbf G(g_\ell,\Lambda)$ is a Parseval frame.
\end{theorem}

For the explicit lattice $\mathbb{Z}^d\times\bigl(\mathbb{Z}^{d-1}\times q^{-1}\mathbb{Z}\bigr)$ with $q \in \N$, the criterion of Enstad, Thiel, and Vilalta yields the existence of a frame $\mathbf{G}(g,\Lambda)$ with window $g \in \mathcal{S}(\Rd)$ whenever $q>d$ \cite[Example 6.3]{enstad2025criteria}. It follows from Theorem \ref{thm:main} that the complementary range $1\leq q\leq d$ rules out the existence of any $g \in \mathcal{S}(\Rd)$ generating a Gabor frame. In particular, this answers a problem raised in \cite[p.~10]{enstad2025criteria} on the existence of Schwartz functions generating a Gabor frame that violate the inequality \eqref{eq:nLambda_inequality}.

\subsection{}
In the situation of Theorem \ref{thm:main}, if $rN>d$, then it follows from
\cite[Theorem D]{enstad2025criteria} that there exist Schwartz functions
$g_1,\dots,g_r$ such that
$
\bigcup_{\ell=1}^r \mathbf G(g_\ell,\Lambda)
$
is a frame. Theorem \ref{thm:main} differs from the result in
\cite{enstad2025criteria} in two ways. First, we show that the frame can be
chosen to be Parseval. This property is particularly important in applications,
where the frame bounds determine the condition number of the associated frame
expansion. Second, in Section \ref{sec:construction} we provide, for every
dimension $d>1$, a rather explicit construction of a family of Schwartz functions $g_\ell$ yielding a Gabor Parseval frame. The geometrical idea of our construction is outlined in Section \ref{sec:geom_idea}. By comparison, the existence result
in \cite{enstad2025criteria} is based on methods from $C^*$-algebras and does
not provide further information on the structure of the Schwartz functions
$g_\ell$.

\subsection{}
We conclude this section by introducing a key ingredient in the proof of Theorem \ref{thm:main}, related to the study of zeros of quasiperiodic functions. In what follows we call a function $s : \Rd \times \Rd \to \C$ Zak-quasiperiodic (or simply quasiperiodic) if it satisfies the relation \begin{equation}\label{eq:quasi_periodicity}
    s(u+n,v+m)=e^{2\pi i\,n\cdot v}s(u,v)
\end{equation}
for every $u,v\in\mathbb R^d$ and every
$n,m\in\mathbb Z^d$. The notion of (Zak-)quasiperiodicity stems from the fact that the Zak transform of every $g \in \ltd$ is quasiperiodic.

The structure of the zero set of the Zak transform is a central topic in the analysis of Gabor frames. For instance, a result of Vinogradov and Ulitskaya establishes the existence of a unique zero for the Zak transform of totally positive functions \cite{VINOGRADOV201755}; this fact plays a significant role in the frame set conjecture of totally positive functions \cite{grochenig2023totally,GroechenigRomeroStoeckler}. It is well-known that every continuous quasiperiodic function has a zero. This result was proven by Boon and Zak \cite{boon1981amplitudes} and independently by Janssen \cite{janssen1982bargmann}. A topological proof of the result was obtained by Auslander and Tolimieri \cite[Theorem II.2]{auslander2006abelian} in the context of nilmanifolds. A proof based on Chern-Weil theory was given by Enstad \cite[Example 5.6]{enstad2020balian}.
Extensions of the result to the setting of locally compact abelian groups have been developed in \cite{enstad2022deformations, kaniuth1998zeros}.

The proof of Theorem \ref{thm:main} makes use (cf.~Lemma \ref{lma:Zak_frame}) of the following statement on common zeros of multiple quasiperiodic functions, which we prove in Section \ref{sec:zak_zero} using a topological argument involving Chern classes.

\begin{theorem}\label{thm:zak_zero}
Let $d,N\in \N$, and suppose that
$$
s_1,\dots,s_N:\mathbb R^d \times \mathbb R^d \to \mathbb C
$$
are continuous Zak-quasiperiodic functions. If $N \leq d$, then $s_1,\dots,s_N$ have a common zero, i.e., there exists $(u_0,v_0)\in\mathbb R^d\times\mathbb R^d$ such that
$$
s_1(u_0,v_0)=s_2(u_0,v_0)=\cdots=s_N(u_0,v_0)=0.
$$
If $N>d$, then the conclusion is false: there exist smooth Zak-quasiperiodic functions
$s_1,\dots,s_N :\mathbb R^d \times \mathbb R^d  \to \mathbb C$ (not all equal to zero) such that $s_1,\dots,s_N$ do not have a common zero.
\end{theorem}

\begin{figure}[h]
\centering \hspace*{-0.5cm}
   \includegraphics[width=12.5cm]{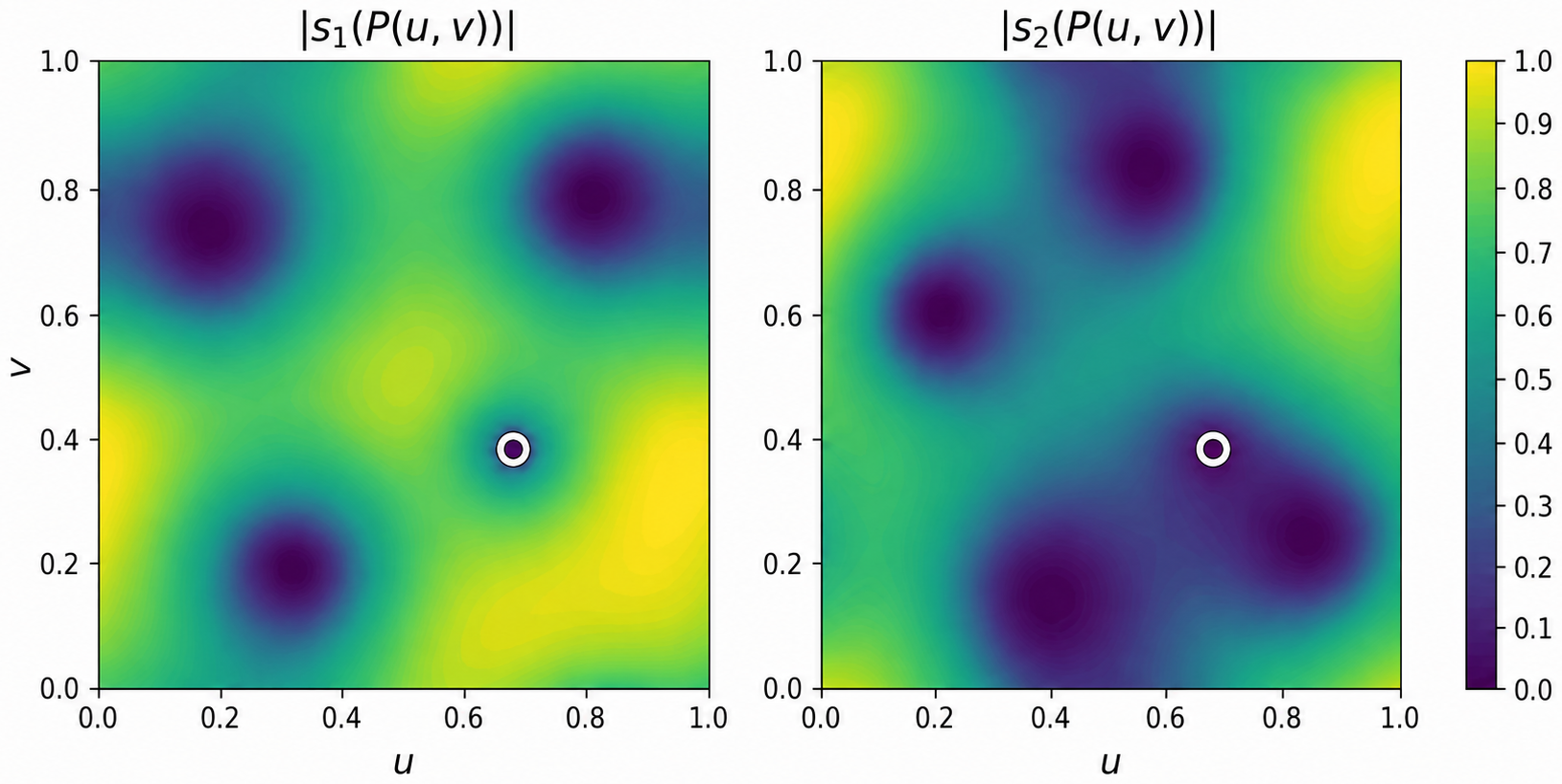}
\caption{Visualization of Theorem \ref{thm:zak_zero} in the case $d=N=2$. The figure shows the heatmap of two quasiperiodic functions $s_1,s_2 : \R^4 \to \C$ on the intersection of the 4-cube $[0,1]^4$ with the plane $\mathcal{P}$, parametrized by $P : \R^2 \to \R^4$ with $P(u,v) = (u,v,v,1-u)$. The functions $s_1$ and $s_2$ are chosen as products of the theta function $\vartheta_3$ and certain sine functions. The white circle marks a common zero of $s_1$ and $s_2$. We refer to the GitHub repository for the precise definitions of $s_1,s_2$ and the Python code used to create the figure.}
\label{fig:intro_plot}
\end{figure}

\subsection*{Usage of Large Language Models} 

Large language models, used under the authors guidance, played a significant role in the development of this work.

We began with the property that for the lattice $\Z^{2d}$, no function $g \in C_Z(\mathbb{R}^d)$ can generate a Gabor frame $\mathbf{G}(g,\Z^{2d})$. Indeed, the Zak transform $Zg$ is continuous and quasiperiodic, and therefore must vanish at some point.

The remaining open case for the existence of Schwartz Gabor frames along lattices was that of rational lattices. If $\Lambda$ is a lattice containing $\Z^{2d}$ as a sublattice of index $N$, Lemma \ref{lma:Zak_frame} shows that the frame property is equivalent to the non-vanishing of a sum of squared moduli of Zak transforms, hence a sum of squared moduli of quasiperiodic functions.

For $N=2$ a contradiction to the lower frame property can be obtained by showing that two quasiperiodic functions have a common zero. Auslander and Tolimieri gave in \cite[p.~18, Theorem II.2]{auslander2006abelian} a topological proof that a continuous quasiperiodic function has a zero. In the time--frequency analysis literature, this approach is somewhat nonstandard, since more elementary proofs of this fact are available. In the approach of Auslander and Tolimieri, a continuous quasiperiodic function is interpreted as a section of a nontrivial complex line bundle over the 2-dimensional torus $\T^2$, and the final conclusion is obtained via a contradiction involving winding numbers. Under an additional smoothness assumption the winding number step can be replaced by degree theory \cite{milnor1965topology,guillemin2025differential} to obtain a common zero result for two quasiperiodic functions.

At this stage, we used GPT-5.4 to explore whether methods from homology theory could provide criteria for common zeros of quasiperiodic functions if a smoothness assumption is replaced by a continuity assumption. Building on the perspective of winding numbers and degree theory, and with assistance from GPT-5.4, we identified a possible approach to proving a common-zero result for two continuous quasiperiodic functions using Chern classes. After substantial refinement, we obtained that $N$ quasiperiodic functions in $\R^d \times \R^d$ must have a common zero whenever $N \leq d$. This led us to the conclusion that for
$
\Lambda=\mathbb{Z}^d\times\bigl(\mathbb{Z}^{d-1}\times q^{-1}\mathbb{Z}\bigr),
$
there exists no Schwartz function $g$ such that $\mathbf{G}(g,\Lambda)$ is a frame. We then generalized this argument to obtain Theorem \ref{thm:main}.

For the case $q>d$, where Schwartz windows do exist, we developed a construction which was again based on the theory of Chern classes. This was then refined into an elementary construction yielding a Parseval frame.

Finally, we remark that the non-existence result of Schwartz Gabor frames for the lattices in Theorem \ref{thm:main} has been formally verified in Lean 4 for every $d>1$ (see the Appendix).

In this project, GPT-5.4 was used mainly for mathematical exploration, while Claude Opus 4.7 assisted with the Lean formalization. The authors independently verified all final statements and substantially rewrote every part that was influenced by LLM-generated material.

\section{Preliminaries}

This section is devoted to a discussion of preliminary concepts and results that will be used throughout the remainder of the article. These include properties of the Zak transform and the Wiener space, as well as topological preliminaries including Chern classes.

\subsection{The Zak transform}

In what follows we define for $x,\omega  \in \Rd$ the translation and modulation operators via $T_xf(t):=f(t-x)$ and $M_\omega f(t):=e^{2\pi i\,\omega\cdot t}f(t)$, respectively. Moreover, we define $\pi(x,\omega):=M_\omega T_x$.
For a Schwartz function $f\in\mathcal S(\mathbb R^d)$, we define the Zak transform $Zf$ by
$$
Zf(u,v):=\sum_{n\in\mathbb Z^d} f(u-n)e^{2\pi i\,n\cdot v},
\quad u,v\in\mathbb R^d.
$$
It is well-known that $Z$ extends from $\mathcal{S}(\Rd)$ to a unitary operator $Z:L^2(\mathbb R^d)\to L^2([0,1]^{2d})$ so that $Zf$ is quasiperiodic, i.e.,
$$
Zf(u+n,v+m)
=
e^{2\pi i\,n\cdot v}Zf(u,v),
\quad n,m\in\mathbb Z^d, \quad u,v \in \Rd.
$$
Moreover, for every $(a,b)\in\mathbb R^d\times\mathbb R^d$, one has
\begin{equation}\label{eq:Zak_TS_property}
    Z(\pi(a,b)f)(u,v)
=
e^{2\pi i\,b\cdot u}Zf(u-a,v-b).
\end{equation}
In particular, for every $a,b \in \R^d$ the space $C_Z(\mathbb R^d)$ is invariant under $\pi(a,b)$. The frame property of a multi-window Gabor system is directly related to lower and upper bounds on the Zak transform. Precisely, the following holds \cite[p.~159]{Groechenig}.

\begin{lemma}\label{lma:Zak_frame}
    For $g_1,\dots,g_M\in \ltd $ the system
$
\bigcup_{j=1}^M\mathbf G(g_j,\Z^{2d})
$
is a frame for $L^2(\mathbb R^d)$ with frame constants $A,B>0$ if and only if
$$
A\le \sum_{j=1}^M |Zg_j(u,v)|^2\le B
$$
for almost every $(u,v)\in[0,1]^{2d}$.
\end{lemma}

We also make use of the following classical result \cite[Theorem 8.2.5]{Groechenig}, which relates the Schwartz space to smooth quasiperiodic functions.

\begin{lemma}\label{lma:Zak_bijection}
The Zak transform is a linear bijection from $\mathcal{S}(\mathbb{R}^d)$ onto the space of smooth quasiperiodic functions on $\R^{2d}$.
\end{lemma}

\subsection{The Wiener space}\label{sec:wiener}

A function $f \in L^\infty(\Rd)$ belongs to the Wiener space $W(\Rd)$ if
$$
\sum_{n \in \Z^d} \| f(\cdot + n) \|_{L^\infty([0,1]^d)} < \infty.
$$
Moreover, $f$ belongs to the Fourier invariant Wiener space $W_0(\Rd)$ if both $f$ and $\hat f$ belong to $W(\Rd)$, where $\hat f(\xi) = \int_{\Rd} f(x) e^{-2\pi i x \cdot \xi} \, dx$ denotes the Fourier transform of $f$. It is well-known that $W_0(\Rd) \subseteq C_Z(\Rd)$ \cite{Groechenig}. The next example shows that this inclusion is in fact proper.

\begin{example}
    Consider the function $a : \R \to \R$ defined by
    \begin{equation}\label{eq:sine_series}
        a(t) = \sum_{n=2}^\infty \frac{1}{n \log n} \sin (2\pi nt).
    \end{equation} 
    By a classical theorem of Chaundy and Jolliffe \cite{chaundy1917uniform} on uniform convergence of monotone sine series, it follows that the series in \eqref{eq:sine_series} converges uniformly and defines a 1-periodic continuous function with $a(0)=a(1)=0$. Let $f \in W(\R)$ be defined by $f = \chi_{[0,1]} a$ where $\chi_S$ denotes the characteristic function of a set $S \subseteq \R$. Then it follows from the definition of the Zak transform that for $x=m+r$ with $m \in \Z$ and $r \in [0,1)$, we have
    $$
    Zf(x,\omega) = a(r) e^{2\pi i m \omega}.
    $$
    Hence, $f \in C_Z(\R)$. On the other hand, the Fourier transform of $f$ evaluated at $n \in \N$ with $n \geq 2$ satisfies
    $
    \hat f(n) = \frac{1}{2i n \log n}
    $
    and therefore $f \not \in W_0(\R)$.

    For $d > 1$ we consider the product
    $$
    F(x_1, \dots, x_d) := f(x_1) \prod_{j=2}^d \eta(x_j)
    $$
    with $\eta(t) = \sin(\pi t) \chi_{[0,1]}(t)$. Then $F \in C_Z(\Rd)$ but
    $$
    \hat F(n,0,\dots, 0) = \hat f(n) \hat \eta(0)^{d-1}, \quad \hat \eta(0) = \frac 2 \pi \neq 0.
    $$
    The same argument as before implies that $F \not \in W_0(\Rd)$.
\end{example}

Finally, $W_0(\Rd)$ contains the Feichtinger algebra $S_0(\Rd)$ as a subspace and this subspace is proper according to Losert's theorem \cite[Theorem 2]{losert1980characterization}. For the fact that the Schwartz space $\mathcal{S}(\Rd)$ is a proper subspace of $S_0(\Rd)$, we refer to Feichtinger's original paper \cite{feichtinger1981new}.

\subsection{Chern classes}

Throughout this section, all vector
bundles are complex vector bundles unless explicitly stated otherwise.

\medskip
Let $X$ be a topological space. For an Abelian group $A$,
$
H^q(X;A)
$
denotes the $q$-th singular cohomology group of $X$ with coefficients in $A$.
The cohomology groups form a graded ring under the cup product
$$
\smile:
H^p(X;A)\times H^q(X;A)\to H^{p+q}(X;A),
$$
which is graded-commutative: $\alpha\smile\beta
=
(-1)^{pq}\beta\smile\alpha$.
If $X$ is a smooth manifold, then de Rham cohomology identifies
$H^q(X;\mathbb R)$ with closed smooth $q$-forms modulo exact $q$-forms.
Under this identification, the cup product corresponds to the wedge product
of differential forms. Thus, if $\omega$ is a closed differential form, we
write
$
[\omega]\in H^q(X;\mathbb R)
$
for its cohomology class.

\medskip
A complex vector bundle of rank $r$ over a topological space $X$ is a map
$
\pi:E\to X
$
such that for every $x\in X$, the fiber
$
E_x:=\pi^{-1}(x)
$
is a complex vector space isomorphic to $\mathbb C^r$, and such that locally
over $X$ the bundle is isomorphic to a product:
$
\pi^{-1}(U)\cong U\times\mathbb C^r.
$
A section of $E$ is a continuous map
$
S:X\to E
$
such that
$
\pi(S(x))=x
$
for every $x\in X$.
The trivial complex line bundle over $X$ is denoted by
$
\underline{\mathbb C}:=X\times\mathbb C\to X.
$
More generally, the trivial rank-$r$ complex vector bundle is
$
\underline{\mathbb C}^r:=X\times\mathbb C^r\to X.
$
If $E\to X$ and $F\to X$ are complex vector bundles, their direct sum is the
bundle
$
E\oplus F\to X
$
whose fiber over $x$ is
$
(E\oplus F)_x=E_x\oplus F_x.
$

\medskip
To every complex vector bundle $E\to X$ of rank $r$, one associates cohomology
classes
$$
c_k(E)\in H^{2k}(X;\mathbb Z),
\quad k=0,1,\dots,r,
$$
called the \emph{Chern classes} of $E$. The class $c_k(E)$ has cohomological degree
$2k$.
The total Chern class is the formal sum
$$
c(E):=1+c_1(E)+c_2(E)+\cdots+c_r(E).
$$
We record the following properties of Chern classes \cite[Chapter 14]{milnor1974characteristic}.

\begin{lemma}
Let $E,F\to X$ be complex vector bundles.
\begin{enumerate}
    \item If $E$ has complex rank $r$, then
    $
    c_k(E)=0
    $
    for all $k > r$.

    \item The zeroth Chern class is
    $
    c_0(E)=1\in H^0(X;\mathbb Z).
    $
    \item The trivial complex line bundle has total Chern class
    $
    c(\underline{\mathbb C})=1.
    $
    \item If $L\to X$ is a complex line bundle, then
    $
    c(L)=1+c_1(L).
    $
    \item The Whitney product formula holds:
    $
    c(E\oplus F)=c(E)c(F).
    $
\end{enumerate}
\end{lemma}

We also make use of the following statement from \cite{milnor1974characteristic}.

\begin{lemma}
Let $E\to X$ be a complex vector bundle of rank $N$ over a compact Hausdorff
space $X$. If $E$ admits a nowhere-zero continuous section, then
$
c_N(E)=0.
$
\end{lemma}

Let $L\to M$ be a smooth complex line bundle over a smooth manifold $M$, and let $\nabla$ be a connection on $L$.
In a local trivialization of $L$, a connection has the form
$
\nabla=d+A,
$
where $A$ is a complex-valued one-form and its curvature is
$
F_\nabla=dA.
$
For a line bundle, this is a globally defined complex-valued two-form.
With the convention used in the upcoming sections, the real first Chern class of $L$ is
represented in de Rham cohomology by
$$
c_1(L)_{\mathbb R}
=
\left[\frac{i}{2\pi}F_\nabla\right]
\in H^2(M;\mathbb R).
$$
In other words,
$
\frac{i}{2\pi}F_\nabla
$
is a closed real two-form whose de Rham cohomology class is the image of
$c_1(L)\in H^2(M;\mathbb Z)$ in real cohomology.

\section{Proof of Theorem \ref{thm:zak_zero}}\label{sec:zak_zero}

\subsection{Common zero of quasiperiodic functions}
We start by showing the first part of Theorem \ref{thm:zak_zero}, i.e., if $s_1,\dots,s_N$ are continuous quasiperiodic functions on $\Rd$ with $N \leq d$ then they have a common zero.

\begin{proof}[Proof of the first part of Theorem \ref{thm:zak_zero}]    

In the following, we denote an element of $\Rd \times \Rd$ by $(u,v)$. Assume by contradiction that
\begin{equation}\label{eq:star}
    (s_1(u,v),\dots,s_N(u,v))\neq 0, \quad (u,v)\in\mathbb R^d\times\mathbb R^d.
\end{equation}

\smallskip
\noindent\textbf{Step 1: the quasiperiodic line bundle.}
Let
$
\mathbb T^{2d}
:=
(\mathbb R^d/\mathbb Z^d)_u\times(\mathbb R^d/\mathbb Z^d)_v
$
and define a complex line bundle
$
\mathcal L \to \mathbb T^{2d}
$
as follows. Start with the trivial bundle
$
\mathbb R^{2d}\times \mathbb C\to \mathbb R^{2d}
$
and let $\mathbb Z^d\times\mathbb Z^d$ act on $\mathbb R^{2d}\times \mathbb C$ by
$$
(n,m)\cdot (u,v,z)
:=
\bigl(u+n,\ v+m,\ e^{2\pi i\,n\cdot v}z\bigr),
$$
where $n,m\in\mathbb Z^d$.
This is a well-defined group action since for every pair $(n,m),(n',m')\in
\mathbb Z^d\times\mathbb Z^d$, one has
$$
\begin{aligned}
(n,m)\cdot\bigl((n',m')\cdot(u,v,z)\bigr)
&=
(n,m)\cdot
\bigl(u+n',v+m',e^{2\pi i\,n'\cdot v}z\bigr)  \\
&=
\bigl(u+n+n',v+m+m',
e^{2\pi i\,n\cdot(v+m')}e^{2\pi i\,n'\cdot v}z\bigr)  \\
&=
\bigl(u+n+n',v+m+m',
e^{2\pi i\,(n+n')\cdot v}z\bigr) \\
&=
(n+n',m+m')\cdot(u,v,z).
\end{aligned}
$$
Thus, we may form the quotient bundle
$$
\mathcal L
:=
(\mathbb R^{2d}\times\mathbb C)/(\mathbb Z^d\times\mathbb Z^d)
\to
\mathbb T^{2d}.
$$
A continuous section of $\mathcal L$ is equivalently a continuous function
$
\sigma:\mathbb R^{2d}\to\mathbb C
$
satisfying
\begin{equation}\label{eq:333}
    \sigma(u+n,v+m)
=
e^{2\pi i\,n\cdot v}\sigma(u,v)
\end{equation}
for all $n,m\in\mathbb Z^d$.
Indeed, such a function defines a section by
$$
[u,v]\longmapsto [(u,v,\sigma(u,v))],
$$
and condition \eqref{eq:333} is exactly the condition that this definition is independent of
the chosen representative $(u,v)$ of the point $[u,v]\in\mathbb T^{2d}$.
By assumption, each $s_\ell$ satisfies \eqref{eq:333}. Hence, each $s_\ell$ defines a continuous
section of $\mathcal L$ and therefore
$$
S:=(s_1,\dots,s_N)
$$
defines a continuous section of the rank-$N$ complex vector bundle
$$
E_N:=\mathcal L^{\oplus N} \to \mathbb T^{2d}.
$$
Assumption \eqref{eq:star} says exactly that $S$ is a nowhere-zero section of
$E_N$.

\smallskip
\noindent\textbf{Step 2: computation of the first Chern class of $\mathcal L$.}
For every integer $j \in \{ 1,\dots,d \}$, define $\alpha_j \in H^2(\mathbb T^{2d};\mathbb R)$ via
$$
\alpha_j:=[du_j\wedge dv_j].
$$
We now prove that
$
c_1(\mathcal L)_{\mathbb R}
=
\alpha_1+\cdots+\alpha_d.
$
To do so, we start by computing the Chern class using a connection. Although the sections
$s_\ell$ are only assumed continuous, the bundle $\mathcal L$ itself is a smooth
complex line bundle, so its Chern class may be computed using any smooth connection.
Let $\sigma:\mathbb R^{2d}\to\mathbb C$ be a smooth function satisfying
$$
\sigma(u+n,v+m)=e^{2\pi i\,n\cdot v}\sigma(u,v),
$$
and define
$$
\nabla\sigma
:=
d\sigma
-
2\pi i
\left(\sum_{j=1}^d u_j\,dv_j\right)\sigma.
$$
We check that $\nabla$ descends to a connection on $\mathcal L$. Indeed,
$$
\begin{aligned}
(\nabla\sigma)(u+n,v+m)
&=
d\!\left(e^{2\pi i\,n\cdot v}\sigma(u,v)\right)
-
2\pi i\sum_{j=1}^d (u_j+n_j)\,dv_j\,
e^{2\pi i\,n\cdot v}\sigma(u,v) \\
&=
e^{2\pi i\,n\cdot v}
\left(
d\sigma
+
2\pi i\sum_{j=1}^d n_j\,dv_j\,\sigma
-
2\pi i\sum_{j=1}^d u_j\,dv_j\,\sigma
-
2\pi i\sum_{j=1}^d n_j\,dv_j\,\sigma
\right) \\
&=
e^{2\pi i\,n\cdot v}
\left(
d\sigma
-
2\pi i\sum_{j=1}^d u_j\,dv_j\,\sigma
\right) \\
&=
e^{2\pi i\,n\cdot v}(\nabla\sigma)(u,v).
\end{aligned}
$$
Thus, $\nabla$ defines a
connection on $\mathcal L$.
The connection one-form in the above trivialization over $\mathbb R^{2d}$ is
$$
A
=
-2\pi i\sum_{j=1}^d u_j\,dv_j.
$$
Since $\mathcal L$ has rank one, the curvature is
$$
F_\nabla=dA
=
-2\pi i\sum_{j=1}^d du_j\wedge dv_j.
$$
For a complex line bundle, the real first Chern class is represented by
$
\frac{i}{2\pi}F_\nabla,
$
which implies that
$$
c_1(\mathcal L)_{\mathbb R}
=
\left[\frac{i}{2\pi}F_\nabla\right]
=
\sum_{j=1}^d [du_j\wedge dv_j]
=
\alpha_1+\cdots+\alpha_d.
$$

\smallskip
\noindent\textbf{Step 3: the top Chern class of $E_N$.}
Since
$
E_N=\mathcal L^{\oplus N},
$
the Whitney product formula gives
$
c(E_N)=c(\mathcal L)^N.
$
Because $\mathcal L$ is a line bundle,
$
c(\mathcal L)=1+c_1(\mathcal L).
$
Hence,
$
c_N(E_N)=c_1(\mathcal L)^N.
$
Passing to real cohomology and using the computation from above, we get
$$
c_N(E_N)_{\mathbb R}
=
(\alpha_1+\cdots+\alpha_d)^N.
$$
We now show that this class is nonzero. To do so, consider the $2N$-dimensional subtorus $Y$ of $\T^{2d}$ defined by
$$
Y
:=
\{u_{N+1}=\cdots=u_d=v_{N+1}=\cdots=v_d=0\}
\subset \mathbb T^{2d}.
$$
We have the properties
$
Y\cong(\mathbb T^2)^N,
$
$
\alpha_j|_Y=0
$
for $j > N$, and hence
\begin{equation}\label{eq:777}
    \left.c_N(E_N)_{\mathbb R}\right|_Y
=
(\alpha_1+\cdots+\alpha_N)^N.
\end{equation}
The cohomology ring of the torus is the exterior algebra on the degree-one classes
$[du_1],\dots,[du_d],[dv_1],\dots,[dv_d]$. In particular,
$$
\alpha_j^2
=
[du_j\wedge dv_j]\smile[du_j\wedge dv_j]
=
0.
$$
Since the classes $\alpha_j$ have degree $2$, they commute with one another. Hence
the multinomial expansion of \eqref{eq:777} gives
\begin{equation}\label{eq:888}
    (\alpha_1+\cdots+\alpha_N)^N
=
N!\,\alpha_1\smile\cdots\smile\alpha_N.
\end{equation}
Equivalently,
$$
(\alpha_1+\cdots+\alpha_N)^N
=
N!\,[du_1\wedge dv_1\wedge\cdots\wedge du_N\wedge dv_N].
$$
Giving $Y$ its standard orientation, we obtain the identity
$$
\int_Y
du_1\wedge dv_1\wedge\cdots\wedge du_N\wedge dv_N
=
1.
$$
Thus, by \eqref{eq:777} and \eqref{eq:888} we obtain
$$
\int_Y
\left.c_N(E_N)_{\mathbb R}\right|_Y
=
N!\neq 0.
$$
The latter identity implies that
$
c_N(E_N)_{\mathbb R}\neq 0.
$
In particular, we have shown that
\begin{equation}\label{eq:999}
    c_N(E_N)\neq 0
\end{equation}
in $H^{2N}(\mathbb T^{2d};\mathbb Z)$.

\smallskip
\noindent\textbf{Step 4: contradiction from a nowhere-zero section.}
By Step 1, the functions $s_1,\dots,s_N$ define a continuous section
$
S=(s_1,\dots,s_N)
$
of
$
E_N=\mathcal L^{\oplus N}.
$
By the contradiction assumption, $S$ is nowhere zero.
Since $\mathbb T^{2d}$ is a compact Hausdorff space, every complex vector bundle over
$\mathbb T^{2d}$ admits a continuous Hermitian metric. The nowhere-zero section $S$
spans a trivial complex line subbundle
$
\mathbb C S\subset E_N.
$
Taking its orthogonal complement with respect to a continuous Hermitian metric gives
a rank-$(N-1)$ complex vector bundle $E'$ such that
$
E_N\cong \underline{\mathbb C}\oplus E'.
$
Therefore, by the Whitney product formula,
$$
c(E_N)
=
c(\underline{\mathbb C})c(E')
=
c(E').
$$
Since $E'$ has complex rank $N-1$, its degree-$2N$ Chern class vanishes,
$
c_N(E')=0.
$
Hence, we have
$
c_N(E_N)=0,
$
contradicting \eqref{eq:999}.
\end{proof}

\subsection{Quasiperiodic functions with no common zero}
Next, we prove the second part of the statement of Theorem \ref{thm:zak_zero} by an explicit construction, i.e., if $N>d$ we construct smooth quasiperiodic functions
$
s_1,\dots,s_N:
\mathbb R^d_u\times\mathbb R^d_v\to\mathbb C
$
such that $s_1,\dots,s_N$ have no common zero. 

In order to construct such functions we let $\eta\in C^\infty(\mathbb R)$ be given by
$$
\eta(t)
=
\begin{cases}
e^{-1/t}, & t>0,\\
0, & t\le 0,
\end{cases}
$$
and define $\varphi,\psi \in C_c^\infty(\R)$ via
$$
\varphi(t)
=
\eta\!\left(t+\tfrac12\right)
\eta\!\left(\tfrac12-t\right), \quad \psi(t)
=
\eta(t)\eta(1-t).
$$
Then for every $t \in (-\frac12,\frac12)$ we have $\varphi(t)>0$ and for every $y \leq -\frac12$ or $y>\frac12$ we have $\varphi(y)=0$.
Similarly, $\psi(t)>0$ for $0<t<1$ and $\psi(y)=0$ for $y \leq 0$ or $y \geq 1$.

Let $a,b$ be the Zak transform of $\varphi$ and $\psi$, respectively,
$$
a(u,v)
=
\sum_{k\in\mathbb Z}
\varphi(u-k)e^{2\pi i k v}, \quad b(u,v)
=
\sum_{k\in\mathbb Z}
\psi(u-k)e^{2\pi i k v}.
$$
Since $\varphi$ and $\psi$ have compact
support, the above sums are locally finite.
Therefore, $a$ and $b$ are smooth quasiperiodic functions.

\begin{lemma}
The functions $a$ and $b$ have no common zero.
\end{lemma}

\begin{proof}
Fix $u\in\mathbb R$.
If $u\in\mathbb Z$, then the term $k=u$ appears in the sum defining $a$,
and
$
\varphi(u-k)=\varphi(0)>0.
$
Moreover, for $k\neq u$, the value $u-k$ is a nonzero integer, so
$
\varphi(u-k)=0.
$
Hence,
$$
a(u,v)=\varphi(0)e^{2\pi i u v}\neq 0.
$$
If $u\notin\mathbb Z$, then there is a unique integer $k$ such that
$
0<u-k<1.
$
For this $k$, we have
$
\psi(u-k)>0.
$
For every other integer $r\neq k$, we have
$
u-r\notin (0,1),
$
so
$
\psi(u-r)=0.
$
Therefore,
$$
b(u,v)=\psi(u-k)e^{2\pi i k v}\neq 0.
$$
Thus, for every $(u,v)\in\mathbb R^2$, at least one of $a(u,v)$ and
$b(u,v)$ is nonzero. Hence, $a$ and $b$ have no common zero.
\end{proof}

Next, we pass from dimension one to dimension $d$. To do so, define for every $j=1,\dots,d$ smooth functions $a_j,b_j:
\mathbb R^d_u\times\mathbb R^d_v\to\mathbb C$ by
$$
a_j(u,v)
=
a(u_j,v_j), \quad b_j(u,v)
=
b(u_j,v_j),
$$
where $u=(u_1,\dots,u_d)$ and $v=(v_1,\dots,v_d)$.
From the one-dimensional quasiperiodicity of $a$ and $b$, we get
\begin{equation}\label{eq:a_b}
    a_j(u+n,v+m)
=
e^{2\pi i n_j v_j}a_j(u,v), \quad b_j(u+n,v+m)
=
e^{2\pi i n_j v_j}b_j(u,v)
\end{equation}
for every $n,m\in\mathbb Z^d$.
Now introduce an auxiliary variable $T \in \C$, and define
$$
P_{u,v}(T)
=
\prod_{j=1}^d
\bigl(a_j(u,v)+T b_j(u,v)\bigr).
$$
Write this polynomial as
$$
P_{u,v}(T)
=
S_0(u,v)+S_1(u,v)T+\cdots+S_d(u,v)T^d.
$$
The coefficients $S_k(u,v)$ satisfy
$$
S_k(u,v)
=
\sum_{\substack{I\subset\{1,\dots,d\}\\ |I|=k}}
\left(
\prod_{j\in I} b_j(u,v)
\right)
\left(
\prod_{j\notin I} a_j(u,v)
\right),
$$
which implies that $S_0,\dots,S_d$
are smooth functions on
$
\mathbb R^d_u\times\mathbb R^d_v.
$

\begin{lemma}
For every $k=0,\dots,d$, the functions $S_k$ are quasiperiodic.
\end{lemma}

\begin{proof}
Using equation \eqref{eq:a_b} we obtain
$$
\begin{aligned}
P_{u+n,v+m}(T)
&=
\prod_{j=1}^d
\bigl(a_j(u+n,v+m)+T b_j(u+n,v+m)\bigr) \\
&=
\prod_{j=1}^d
e^{2\pi i n_j v_j}
\bigl(a_j(u,v)+T b_j(u,v)\bigr) \\
&=
e^{2\pi i n\cdot v}P_{u,v}(T).
\end{aligned}
$$
Comparing coefficients of $T^k$ yields the statement.
\end{proof}

\begin{proof}[Proof of the second part of Theorem \ref{thm:zak_zero}]

By the preceding discussion, the functions $S_0, \dots, S_d$ are smooth and quasiperiodic. It suffices to show that $
S_0,\dots,S_d
$
have no common zero. To do so, fix
$
(u,v)\in\mathbb R^d\times\mathbb R^d
$
and use that for each $j=1,\dots,d$, the pair
$
(a_j(u,v),b_j(u,v))
$
is not equal to $(0,0)$, because $a$ and $b$ have no common zero in
dimension one.
Hence, the linear polynomial
$$
a_j(u,v)+T b_j(u,v)
$$
is not the zero polynomial in $\mathbb C[T]$. Since $\mathbb C[T]$ has no
zero divisors, the product
$$
P_{u,v}(T)
=
\prod_{j=1}^d
\bigl(a_j(u,v)+T b_j(u,v)\bigr)
$$
is also not the zero polynomial.
Therefore, not all coefficients of $P_{u,v}(T)$ vanish. That is, at least one
of
$$
S_0(u,v),S_1(u,v),\dots,S_d(u,v)
$$
is nonzero. Since $(u,v)$ was arbitrary, $S_0,\dots,S_d$ have no common zero. Since $N \geq d+1$, the $N$ quasiperiodic functions $s_1, \dots, s_N$ defined by
\begin{equation}
    s_j = \begin{cases}
        S_{j-1}, & j = 1,\dots, d+1, \\
        0, & j > d+1,
    \end{cases}
\end{equation}
have no common zero, thereby proving the statement.
\end{proof}

\section{Proof of Theorem \ref{thm:main}}

\subsection{Reduction to the canonical lattice}\label{sec:reduction}
For a sublattice $\Gamma'$ of a lattice $\Gamma$, we denote by $[\Gamma:\Gamma']$ the index of $\Gamma'$. In order to prove Theorem \ref{thm:main}, we first reduce to the case $L=\Z^d$.

Suppose that $L=A\Z^d$ with $A \in \mathrm{GL}_d(\Q)$.
Further, let $L^*=A^{-T}\mathbb Z^d$ be the dual lattice of $L$. Define the dilation operator $D_A : \ltd \to \ltd$ via
$
D_A f(t)=|\det A|^{1/2}f(At).
$
Note that $D_A$ is unitary and satisfies
$$
D_A\pi(x,\omega)D_A^{-1}
=
\pi(A^{-1}x,A^T\omega).
$$
Therefore, the Gabor system $\bigcup_{\ell=1}^r \mathbf G(g_\ell,\Lambda)$ is unitarily equivalent to the Gabor system $\bigcup_{\ell=1}^r \mathbf G(D_A g_\ell,\Lambda')$ where
$$
\Lambda'
:=
\{Q\lambda:\lambda\in\Lambda\}, \quad Q = \mathrm{diag}(A^{-1},A^T).
$$
The image of $L\times L^*$ under $Q$ is $\Z^d\times\mathbb Z^d$ and we have the index identity
$
[\Lambda:L\times L^*]=[\Lambda':\mathbb Z^d\times\mathbb Z^d]=N.
$
Moreover, if $f \in C_Z(\Rd)$ then the property that $A$ is an element of $\mathrm{GL}_d(\Q)$ implies that $D_Af \in C_Z(\Rd)$.

Summarizing, it suffices to prove the theorem for
$
L=\mathbb Z^d.
$
From now on, set
$$
H:=\mathbb Z^d\times\mathbb Z^d
$$
and assume that $\Lambda \subseteq \R^{2d}$ is a lattice such that $H\subseteq\Lambda$ is a sublattice of index
$
[\Lambda:H]=N.
$

\subsection{}
We start by proving the necessity of the condition $rN > d$, that is, if $\Lambda$ is given as in Theorem \ref{thm:main}, then the existence of $g_1,\dots,g_r\in C_Z(\mathbb R^d)$ giving a frame of the form
$
\bigcup_{\ell=1}^r \mathbf G(g_\ell,\Lambda)
$
implies that $rN > d$.

\begin{proof}[Proof of the forward direction of Theorem \ref{thm:main}]
Choose representatives
$
\gamma_1,\dots,\gamma_N\in\Lambda
$
for the finite quotient group
$
\Lambda/H
$
so that $\Lambda$ is given as a disjoint union of the form
$$
\Lambda
=
\bigcup_{j=1}^N(\gamma_j+H).
$$
Write
$
\gamma_j=(a_j,b_j)
$
and
define for $g_1,\dots,g_r \in C_Z(\Rd) $ functions $h_{\ell,j}$ by 
$$
h_{\ell,j}:=\pi(\gamma_j)g_\ell,
\quad
1\le \ell\le r,\quad 1\le j\le N.
$$
The system
$
\bigcup_{\ell=1}^r\mathbf G(g_\ell,\Lambda)
$
coincides, up to unimodular scalar factors, with the system
$$
\bigcup_{\ell=1}^r\bigcup_{j=1}^N\mathbf G(h_{\ell,j},H).
$$
Therefore, $\bigcup_{\ell=1}^r\mathbf G(g_\ell,\Lambda)$ is a frame if and only if $\bigcup_{\ell=1}^r\bigcup_{j=1}^N\mathbf G(h_{\ell,j},H)$ is a frame.
By Lemma \ref{lma:Zak_frame}, $\bigcup_{\ell=1}^r\bigcup_{j=1}^N\mathbf G(h_{\ell,j},H)$ is a frame if and only if
$$
m(u,v)
:=
\sum_{\ell=1}^r\sum_{j=1}^N
|Z h_{\ell,j}(u,v)|^2
$$
satisfies $0<A \leq m(u,v) \leq B < \infty$ for almost every $(u,v)$.
Moreover, the relation \eqref{eq:Zak_TS_property} implies that
$$
Z h_{\ell,j}(u,v)
=
Z(\pi(a_j,b_j)g_\ell)(u,v)
=
e^{2\pi i\,b_j\cdot u}Zg_\ell(u-a_j,v-b_j).
$$

Assume, towards a contradiction, that
$
rN\le d
$
and that there exists functions
$
g_1,\dots,g_r\in C_Z(\mathbb R^d)
$
such that
$
\bigcup_{\ell=1}^r\mathbf G(g_\ell,\Lambda)
$
is a frame. For $1\le\ell\le r$ and $1\le j\le N$, set
$$
s_{\ell,j}(u,v):=Zh_{\ell,j}(u,v).
$$
Because $g_\ell\in C_Z(\mathbb R^d)$, it follows that every
$s_{\ell,j}$ is continuous and quasiperiodic. By Theorem \ref{thm:zak_zero}, there exists
$
(u_0,v_0)\in\mathbb R^d\times\mathbb R^d
$
such that
$
s_{\ell,j}(u_0,v_0)=0
$
for every $\ell=1,\dots,r$ and every $j=1,\dots,N$. Hence,
$
m(u_0,v_0)=0.
$
If $\bigcup_{\ell=1}^r\mathbf G(g_\ell,\Lambda)$ were a frame, then by $m \geq A >0$ almost everywhere and the continuity of each $s_{\ell,j}$ we would have $m \geq A >0$ everywhere, contradicting the property that $m$ has a zero.
\end{proof}

\subsection{Construction idea}\label{sec:geom_idea}
To complete the proof of Theorem \ref{thm:main}, we next show via a construction that if $rN > d$ then there exists $g_1,\dots,g_r\in \mathcal S(\mathbb R^d)$ such that
$
\bigcup_{\ell=1}^r \mathbf G(g_\ell,\Lambda)
$
is a Parseval frame for $L^2(\mathbb R^d)$.

The construction is best understood geometrically on the Zak transform side. Consider the torus $\mathbb R^{2d}/H$, where
$H=\mathbb Z^d\times \mathbb Z^d$.  The representatives
$\gamma_1,\dots,\gamma_N$ of $\Lambda/H$ have a separation constant $\delta$. Choose $\varepsilon>0$ smaller than this separation constant and cover a compact
fundamental domain $P_\Lambda$ of $\Lambda$ by balls $B(c_\alpha,\varepsilon/2)$.  For each
center $c_\alpha$ and each representative $\gamma_j$, we place a small bump at $c_\alpha-\gamma_j$ and quasiperiodize it using the \emph{quasiperiodization operator}
$$
\mathcal P\rho(u,v):=\sum_{n,m\in\mathbb Z^d}e^{2\pi i\,n\cdot v}\,\rho(u-n,v-m).
$$
This yields functions $\varphi_{\alpha,j}$.  The choice of $\varepsilon$ ensures that, whenever
$x\in B(c_\alpha,\varepsilon/2)$, the matrix
$$
    \bigl(\varphi_{\alpha,j}(x-\gamma_i)\bigr)_{1\leq i,j\leq N}
$$
is diagonal with strictly positive diagonal entries and therefore has rank $N$ at every point $x$.
After relabeling the functions as $\varphi_1,\dots,\varphi_K$, the maps
$$
A(x):\mathbb C^K\to \mathbb C^N, \quad A(x)c = \left( \sum_{k=1}^K c_k\varphi_k(x-\gamma_j) \right)_{j=1}^N
$$
are therefore surjective for all $x$.  We then choose linear
combinations
$$
F_\ell=\sum_{k=1}^K c_{\ell,k}\varphi_k, \quad \ell=1,\dots,r,
$$
with coefficients $c_{\ell,k}$. The goal is to choose the coefficients in such a way that we rule out the existence of $x$ such that $F_\ell(x-\gamma_j)=0$ for every $\ell$ and every $j$. Call this the bad set of coefficients. Since
$A(x)$ is surjective, we show that this bad set has codimension $rN$ in the space of coefficient
variables, while $x$ ranges over $d$ complex dimensions. The hypothesis $rN>d$ makes the resulting bad set measure zero by Sard's theorem.  Hence, the coefficients may be chosen so that
$$
M(x) := \sum_{\ell=1}^r\sum_{j=1}^N |F_\ell(x-\gamma_j)|^2
$$
is strictly positive for every $x$.  The function $M$ is $\Lambda$-periodic, so
normalizing
$$
G_\ell(x)=\frac{F_\ell(x)}{\sqrt{M(x)}}
$$
preserves smoothness and quasiperiodicity and gives
$$
\sum_{\ell=1}^r\sum_{j=1}^N |G_\ell(x-\gamma_j)|^2 = 1.
$$
In the following we carry out this idea in a mathematically rigorous way.

\subsection{Construction of a Gabor Parseval frame}\label{sec:construction}
Let $L=\mathbb Z^d$, $H=\mathbb Z^d\times\mathbb Z^d$, and assume that $H\subseteq \Lambda$ has index $[\Lambda:H]=N$.
Choose representatives
$
\gamma_1,\dots,\gamma_N\in\Lambda
$
for the quotient $\Lambda/H$, and write
$$
\gamma_j=(a_j,b_j)\in\mathbb R^d\times\mathbb R^d.
$$

Our goal is to construct smooth quasiperiodic functions
$
G_1,\dots,G_r
$
such that
for every $(u,v)\in\mathbb R^{2d}$ one has
\begin{equation}\label{eq:zak-partition-identity}
\sum_{\ell=1}^r\sum_{j=1}^N
|G_\ell(u-a_j,v-b_j)|^2
=
1.
\end{equation}
In the following we write $x=(u,v)\in\mathbb R^{2d}$. For
$\rho\in C_c^\infty(\mathbb R^{2d})$, we define a quasiperiodized bump-function via
\begin{equation}\label{eq:quasi-periodization}
    \mathcal P\rho(u,v)
:=
\sum_{n,m\in\mathbb Z^d}
e^{2\pi i\,n\cdot v}\rho(u-n,v-m).
\end{equation}
The above sum is locally finite, hence $\mathcal P\rho$ is smooth. It follows from a direct calculation that $\mathcal P\rho$ is quasiperiodic.

\subsection{}
We next construct finitely many smooth quasiperiodic functions
$
\varphi_1,\dots,\varphi_K
$
such that, for every $x\in\mathbb R^{2d}$, the matrix
\begin{equation}\label{eq:evaluation-matrix}
    \bigl(\varphi_k(x-\gamma_j)\bigr)_
{1\le j\le N,\;1\le k\le K}
\end{equation}
has rank $N$.
To do so, let
$$
\operatorname{dist}_H(x,y)
:=
\inf_{h\in H}|x-y-h|
$$
be the distance on the quotient $\mathbb R^{2d}/H$ and define the quantity $\delta>0$ via
$$
\delta :=
\begin{cases}
1, & N=1,\\
\min_{i\neq j}\operatorname{dist}_H(\gamma_i,\gamma_j), & N\geq 2.
\end{cases}
$$
Choose
$$
0<\varepsilon<\frac14\min\{1,\delta\}
$$
and $\chi\in C_c^\infty(\mathbb R^{2d})$ with $\chi\geq 0$, such that
$$
\operatorname{supp}\chi\subset B(0,\varepsilon),
\qquad
\chi(x)>0
\quad\text{for } |x|<\frac{\varepsilon}{2}.
$$
Let $P_\Lambda$ be a compact fundamental parallelepiped for $\Lambda$.
Choose finitely many points $c_1,\dots,c_A\in\mathbb R^{2d}$ such that
$$
P_\Lambda
\subset
\bigcup_{\alpha=1}^A
B\left(c_\alpha,\frac{\varepsilon}{2}\right).
$$
For $1\leq\alpha\leq A$ and $1\leq j\leq N$, define
$$
\rho_{\alpha,j}(x)
:=
\chi\bigl(x-(c_\alpha-\gamma_j)\bigr),
\qquad
\varphi_{\alpha,j}:=\mathcal P\rho_{\alpha,j}.
$$
Clearly, each $\varphi_{\alpha,j}$ is quasiperiodic.

\begin{lemma}
For every $x\in\mathbb R^{2d}$, the matrix
$$
\left(
\varphi_{\alpha,j}(x-\gamma_i)
\right)_{
1\leq i\leq N,\;
(\alpha,j)\in\{1,\dots,A\}\times\{1,\dots,N\}
}
$$
has rank $N$.
\end{lemma}

\begin{proof}
Let $x\in P_\Lambda$ and choose $\alpha$ such that
$$
|x-c_\alpha|<\frac{\varepsilon}{2}.
$$
We will show that the $N\times N$ submatrix
$
\bigl(\varphi_{\alpha,j}(x-\gamma_i)\bigr)_{1\leq i,j\leq N}
$
is diagonal with nonzero diagonal entries.

Write $h=(n_h,m_h)\in H$. A term in
$\varphi_{\alpha,j}(x-\gamma_i)$ can be nonzero only if
$$
x-\gamma_i-h
\in
B(c_\alpha-\gamma_j,\varepsilon),
$$
equivalently,
\begin{equation}\label{eq:bump-overlap-condition}
    \bigl|x-c_\alpha-(\gamma_i-\gamma_j+h)\bigr|<\varepsilon.
\end{equation}
If $i\neq j$, then \eqref{eq:bump-overlap-condition} implies
$$
\operatorname{dist}_H(\gamma_i,\gamma_j)
\leq
|\gamma_i-\gamma_j+h|
<
|x-c_\alpha|+\varepsilon
<
\frac{3\varepsilon}{2}
<
\delta,
$$
a contradiction. Hence, for every $i \neq j$ we have
$
\varphi_{\alpha,j}(x-\gamma_i)=0.
$
If $i=j$, then the term $h=0$ gives
$$
\rho_{\alpha,i}(x-\gamma_i)
=
\chi(x-c_\alpha)>0.
$$
No term with $h\neq 0$ can contribute, because then
$$
|h|
<
|x-c_\alpha|+\varepsilon
<
\frac{3\varepsilon}{2}
<
1,
$$
which is impossible for a nonzero element of $H=\mathbb Z^{2d}$. Therefore
$$
\varphi_{\alpha,i}(x-\gamma_i)=\chi(x-c_\alpha)>0.
$$
Thus, the chosen $N\times N$ submatrix is diagonal with nonzero diagonal
entries, and so has rank $N$.

Now let $x\in\mathbb R^{2d}$. Write $x=y+\lambda$, with
$y\in P_\Lambda$ and $\lambda\in\Lambda$. Since
$\gamma_1,\dots,\gamma_N$ represent $\Lambda/H$, there is a permutation
$\sigma$ of $\{1,\dots,N\}$ and elements $h_i\in H$ such that
$$
\lambda-\gamma_i=-\gamma_{\sigma(i)}+h_i.
$$
Thus,
$$
x-\gamma_i
=
y-\gamma_{\sigma(i)}+h_i.
$$
Write $y=(u,v)$, $\gamma_{\sigma(i)}=(a_{\sigma(i)},b_{\sigma(i)})$, and
$h_i=(n_i,m_i)$. By quasiperiodicity, for every column index
$(\alpha,j)$,
$$
\varphi_{\alpha,j}(y-\gamma_{\sigma(i)}+h_i)
=
e^{2\pi i\,n_i\cdot (v-b_{\sigma(i)})}
\varphi_{\alpha,j}(y-\gamma_{\sigma(i)}).
$$
The scalar is nonzero and independent of the column index. Hence, the matrix at
$x$ is obtained from the matrix at $y$ by a row permutation and multiplication
of rows by nonzero scalars. Therefore, the rank is unchanged. Since the rank at
$y$ is $N$, the rank at $x$ is also $N$.
\end{proof}

Next, we relabel the finitely many functions $\varphi_{\alpha,j}$ as
$
\varphi_1,\dots,\varphi_K.
$
Thus, for every $x\in\mathbb R^{2d}$, the linear map
\begin{equation}\label{eq:Ax-definition}
    A(x):\mathbb C^K\to\mathbb C^N,
\quad
A(x)c
=
\left(
\sum_{k=1}^K c_k\varphi_k(x-\gamma_j)
\right)_{j=1}^N
\end{equation}
is surjective.
For an $N$-element subset $I=\{k_1,\dots,k_N\}\subseteq\{1,\dots,K\}$,
with $k_1<\cdots<k_N$, define
$$
A_I(x)
:=
\bigl(\varphi_{k_s}(x-\gamma_j)\bigr)_
{1\leq j\leq N,\;1\leq s\leq N}
\in \C^{N \times N}.
$$
$A_I(x)$ is the $N\times N$ submatrix of $A(x)$ formed by the
columns indexed by $I$.
If $I^c=\{k_{N+1},\dots,k_K\}$, listed increasingly, we define
$$
A_{I^c}(x)
:=
\bigl(\varphi_{k_s}(x-\gamma_j)\bigr)_
{1\leq j\leq N,\;N+1\leq s\leq K}
\in \mathbb C^{N \times (K-N)}.
$$
For $c\in\mathbb C^K$, we write $c_I\in\mathbb C^N$ and
$c_{I^c}\in\mathbb C^{K-N}$ for the corresponding coordinate subvectors.
Note that
\begin{equation}\label{eq:star_star}
    A(x)c=A_I(x)c_I+A_{I^c}(x)c_{I^c}.
\end{equation}
Below, we will invoke the following consequence of Sard's theorem \cite[Corollary 6.11]{Lee}.

\begin{lemma}\label{lma:sard}
Let $U\subseteq\mathbb R^p$ be $\sigma$-compact, and let
$F:U\to\mathbb R^q$ be the restriction of a smooth map defined on an open
neighborhood of $U$. If $p<q$, then $F(U)$ has Lebesgue measure zero in
$\mathbb R^q$.
\end{lemma}

\subsection{}
Assume now that
$
rN>d.
$
We shall choose constants
$$
c_{\ell,k}\in\mathbb C,
\qquad
1\leq \ell\leq r,\quad 1\leq k\leq K,
$$
and define
\begin{equation}\label{eq:Fl-linear-combination}
    F_\ell(x)
=
\sum_{k=1}^K c_{\ell,k}\varphi_k(x).
\end{equation}
Note that each $F_\ell$ is smooth and quasiperiodic.

\begin{lemma}
There exists a set $\Theta \subset (\mathbb C^K)^r$ of full Lebesgue measure such that if $(c_{\ell,k})_{1\leq \ell\leq r, 1\leq k\leq K} \in \Theta $ then
\begin{equation}\label{eq:no-common-zero-orbit}
    \sum_{\ell=1}^r\sum_{j=1}^N
|F_\ell(x-\gamma_j)|^2
>
0
\end{equation}
for every $x\in\mathbb R^{2d}$.
\end{lemma}

\begin{proof}
Let
$$
C=(c_1,\dots,c_r)\in(\mathbb C^K)^r,
\qquad
c_\ell=(c_{\ell,1},\dots,c_{\ell,K}).
$$
For fixed $x$, the bad condition
$$
F_\ell(x-\gamma_j)=0
\qquad
\text{for all } \ell,j
$$
is equivalent to
$$
A(x)c_\ell=0
\qquad
\text{for every } \ell=1,\dots,r.
$$
Since $A(x)$ is surjective, $\ker A(x)$ has complex dimension $K-N$.

We show that the union of all bad coefficient vectors has measure zero in
$(\mathbb C^K)^r$, identified with $\mathbb R^{2rK}$. For each
$N$-element subset $I\subseteq\{1,\dots,K\}$, let $\Omega_I$ be the open
set of all $x\in\mathbb R^{2d}$ for which the $N\times N$ minor of
$A(x)$ using the columns in $I$ is invertible. The sets $\Omega_I$ cover
$\mathbb R^{2d}$.
Indeed, $A(x)$ has rank $N$ for every $x$, so some $N\times N$
minor of $A(x)$ is invertible.

Fix $I$ and use equation \eqref{eq:star_star} to obtain
$
A_I(x)c_{\ell,I}+A_{I^c}(x)c_{\ell,I^c}=0,
$
or equivalently
\begin{equation}\label{eq:solve-bad-coefficients}
    c_{\ell,I}
=
-A_I(x)^{-1}A_{I^c}(x)c_{\ell,I^c}.
\end{equation}
It follows that all bad coefficients with witness $x\in\Omega_I$ are contained in the
image of the smooth map
$$
\Phi_I:
\Omega_I\times(\mathbb C^{K-N})^r
\to
(\mathbb C^K)^r
$$
defined by taking the free variables $c_{\ell,I^c}$ and then defining
$c_{\ell,I}$ by \eqref{eq:solve-bad-coefficients}. The domain is
$\sigma$-compact and has real dimension
$
2d+2r(K-N).
$
The ambient coefficient space has real dimension
$
2rK.
$
Since $rN>d$, we have
$
2d+2r(K-N)<2rK.
$
By Lemma \ref{lma:sard}, $\Phi_I\bigl(\Omega_I\times(\mathbb C^{K-N})^r\bigr)$
has measure zero. Since there are only finitely many choices of $I$,  the full
bad set has measure zero.
Choosing $C\in(\mathbb C^K)^r$ outside of this bad set guarantees that
\eqref{eq:no-common-zero-orbit} holds for every $x\in\mathbb R^{2d}$.
\end{proof}

Now assume that the $c_{\ell,k}$ belong to the set $\Theta$ given in the previous lemma and define
\begin{equation}\label{eq:M-definition}
    M(x)
:=
\sum_{\ell=1}^r\sum_{j=1}^N
|F_\ell(x-\gamma_j)|^2.
\end{equation}
By construction, we have that
$
M(x)>0
$
for every $x\in\mathbb R^{2d}$. We show that $M$ is $\Lambda$-periodic. To do so, let $\lambda\in\Lambda$ and observe that for
each $j$, there is a permutation $\sigma$ of $\{1,\dots,N\}$ and an
element $h_j\in H$ such that
$$
\lambda-\gamma_j=-\gamma_{\sigma(j)}+h_j.
$$
Since each $F_\ell$ is quasiperiodic, $|F_\ell|^2$ is $H$-periodic.
Therefore,
$$
M(x+\lambda) = \sum_{\ell=1}^r\sum_{j=1}^N
|F_\ell(x+\lambda-\gamma_j)|^2 = \sum_{\ell=1}^r\sum_{j=1}^N
|F_\ell(x-\gamma_{\sigma(j)}+h_j)|^2 = M(x).
$$
In particular, $M$ is $H$-periodic. Since $M$ is smooth, strictly
positive, and $H$-periodic, it has a positive minimum on a fundamental cube
for $H$. Hence, $1/\sqrt{M}$ is smooth.

Finally, set
\begin{equation}\label{eq:Gl-normalized}
    G_\ell(x)
:=
\frac{F_\ell(x)}{\sqrt{M(x)}}.
\end{equation}
Because $M$ is $H$-periodic, each $G_\ell$ is smooth and
quasiperiodic. Also, since $M$ is $\Lambda$-periodic and
$\gamma_j\in\Lambda$, we have
\begin{equation}\label{eq:normalized-zak-identity}
\sum_{\ell=1}^r\sum_{j=1}^N
|G_\ell(x-\gamma_j)|^2  
=
\frac{1}{M(x)}
\sum_{\ell=1}^r\sum_{j=1}^N
|F_\ell(x-\gamma_j)|^2  
=
1.
\end{equation}

\begin{proof}[Proof of the reverse direction of Theorem \ref{thm:main}]
Let $G_\ell$ be defined as in \eqref{eq:Gl-normalized}.
    Invoking Lemma \ref{lma:Zak_bijection}, we may choose
$
g_1,\dots,g_r\in\mathcal S(\mathbb R^d)
$
such that
$
Zg_\ell=G_\ell.
$
Set $h_{\ell,j}:=\pi(\gamma_j)g_\ell$. 
Using the relation \eqref{eq:Zak_TS_property}, we have
$$
Z(\pi(a_j,b_j)g_\ell)(u,v)
=
e^{2\pi i\,b_j\cdot u}
Zg_\ell(u-a_j,v-b_j).
$$
Hence,
$$
|Zh_{\ell,j}(u,v)|^2
=
|G_\ell(u-a_j,v-b_j)|^2.
$$
Combining \eqref{eq:normalized-zak-identity}
with Lemma \ref{lma:Zak_frame} implies that
$
\bigcup_{\ell=1}^r\bigcup_{j=1}^N
\mathbf G(h_{\ell,j},H)
$
is a Parseval frame for $L^2(\mathbb R^d)$. This system is, up to unimodular scalars, equal to $\bigcup_{\ell=1}^r
\mathbf G(g_\ell,\Lambda)$.
\end{proof}

Finally, we note that Theorem \ref{thm:conj} follows from Theorem \ref{thm:main}: for $d > 1$ consider (for example) the lattice $\Lambda = \Lambda_d = \mathbb{Z}^d\times\bigl(\mathbb{Z}^{d-1}\times \frac12 \mathbb{Z}\bigr)$ which has density $D(\Lambda)=2>1$. Then for $L=\Z^d$, we have that $L \times L^* = \Z^d \times \Z^d$ is a sublattice of $\Lambda$ with index $N=2 \leq d$. Applying Theorem \ref{thm:main} with $r=1$ we deduce that there exists no $g \in C_Z(\Rd)$ such that $\mathbf{G}(g,\Lambda)$ is a frame for $\ltd$.

\section{Non-existence in $\mathcal{S}(\Rd)$ and $C_Z(\Rd)$}

If $\Lambda$ is a lattice as in Theorem \ref{thm:main} and $N \leq d$, then, in the single-window case, there does not exist any $g \in C_Z(\Rd)$ such that $\mathbf{G}(g,\Lambda)$ is a frame. It turns out that, in many cases, once an obstruction to the frame property is established for the Schwartz space, it can be lifted to the space $C_Z(\Rd)$. A simple and direct proof of a result of this type can be found below, which constitutes the final statement of the paper.

\begin{proposition}\label{prop:equivalence}
    Let $d \in \N$ and let $L=A\Z^d$ be a lattice with $A \in \mathrm{GL}_d(\Q)$. Further, let $\Lambda \subseteq \R^{2d}$ be a lattice containing $L \times L^*$ as a sublattice of finite index. Then the following are equivalent.
    \begin{enumerate}
        \item There exists $g \in C_Z(\Rd)$ such that $\mathbf{G}(g,\Lambda)$ is a frame.
        \item There exists $g \in \mathcal{S}(\Rd)$ such that $\mathbf{G}(g,\Lambda)$ is a frame.
    \end{enumerate}
\end{proposition}
\begin{proof}
    It suffices to prove that the first statement implies the second one. Moreover, using an analogous reduction as in Section \ref{sec:reduction}, we can assume that $L=\Z^d$ and that $H = \Z^{2d}$ is a sublattice of $\Lambda$ of index $N \in \N$.

    Now suppose that $g \in C_Z(\Rd)$ is a function such that $\mathbf{G}(g,\Lambda)$ is a frame. Choose representatives $z_1, \dots, z_N \in \Lambda$ for $\Lambda / H$. Then $\Lambda$ is given by the disjoint union
    $
    \Lambda = \bigcup_{j=1}^N (z_j + H).
    $
    Moreover, $\mathbf{G}(g,\Lambda)$ is a frame if and only if $\bigcup_{j=1}^N \mathbf{G}(\pi(z_j) g,H)$ is a frame. According to Lemma \ref{lma:Zak_frame} this holds if and only if there exist $A,B>0$ such that
    $$
    A \leq  \sum_{j=1}^N | Z(\pi(z_j)g) |^2 \leq B.
    $$
    Choose $\varepsilon>0$ so small such that $\sqrt{N} \varepsilon < \frac12 \sqrt{A}$ and let $F \in C^\infty(\R^{2d})$ be a smooth Zak-quasiperiodic function such that
    $$
    \| F- Zg \|_{L^\infty(\R^{2d})} < \varepsilon.
    $$
    The latter is possible by the density of smooth Zak-quasiperiodic functions in the space $C_Z(\Rd)$. According to Lemma \ref{lma:Zak_bijection} there exists $f \in \mathcal{S}(\Rd)$ such that $F = Zf$. Since time-frequency shifts act on Zak-transforms by translations and unimodular phases, we have
    $$
    \| Z(\pi(z_j)f) - Z(\pi(z_j)g) \|_{L^\infty(\R^{2d})} < \varepsilon
    $$
    for every $j = 1, \dots, N$. Hence, pointwise we obtain
    $$
    \left (  \sum_{j=1}^N |Z(\pi(z_j) f)(x,\omega)|^2 \right )^{\frac12} \geq \frac12 \sqrt{A}, \quad (x,\omega) \in [0,1]^{2d}.
    $$
    The upper bound is automatic because $Zf$ is smooth on the compact set $[0,1]^{2d}$. We therefore obtain that $\bigcup_{j=1}^N \mathbf{G}(\pi(z_j) f,H)$ is a frame which is equivalent to $\mathbf{G}( f,\Lambda)$ being a frame.
\end{proof}

\section{Appendix: Lean formalization}

The main statement of this work (Theorem \ref{thm:main}) has been verified in the Lean language \cite{lean4}, using the Mathlib library. The result was auto-formalized using GPT 5.5 in Codex with custom prompt skills and Oliver Dressler's MCP \cite{DresslerMCP}. The formalization is available at \url{https://github.com/jaumededios/Schwartz_Gabor_Frames}. The files  \lean{Showcase_CZ.lean} and \lean{Showcase_Schwartz.lean}, produced and extensively reviewed by us, provide a self-contained version of the main statement that was Lean verified, and contain all necessary definitions that were not already in Mathlib.  The reader may also consult  \cite{armstrong2026formalization,bertolini20262, hariharan2026milestone,ilin2026semi} for other recent examples of formalizations involving AI. 

This verification relies on many pre-existing and important mathematical concepts, which exist in Lean thanks to the enormous efforts by the Mathlib community \cite{mathlib2020} in carefully formalizing mathematical objects and ensuring that their meaning is what the mathematical community expects.  These include the definitions of quotient groups, the Fourier transform, and even completely elementary concepts such as the real numbers themselves. The construction of these objects is not detailed in this section, but they have been thoughtfully designed and very carefully reviewed by the Mathlib community.

The goal of this section is to communicate the Lean translation to people who may not be experts in Lean. For the Lean expert, this includes defining \emph{abbreviations}  such as \lean{ℝ^[d]} for $\mathbb R^d$ instead of the more common \lean{(Fin d -> ℝ)}. The result that was formalized in Lean may be stated as follows. 
\begin{theorem*}[Compare with Theorem \ref{thm:main}]
    Let $1\le d,r$ be natural numbers with $d>1$. Let $L$ be a rational lattice in $\mathbb R^d$ with dual lattice $L^*$. Let $\Lambda\supset L\times L^{*}$ be a lattice in $\mathbb R^d \times \mathbb R^d$ of index $N = [\Lambda : L\times L^*]$. Given a family of $r$ Schwartz functions $\mathcal F :=(f_1, \dots, f_r)$, define 
    $$
    \mathbf G(\mathcal F, \Lambda):= 
    \{
    e^{2\pi i w_0 x} f_j(x-x_0), \ \  (x_0, w_0) \in \Lambda,
    \ \ j=1,\dots, r
    \}.
    $$
    Then the following statements are equivalent:
    \begin{enumerate}
        \item $d<r\cdot N$.
        \item There are $r$ Schwartz functions $\mathcal F:=(f_1, \dots, f_r)$ such that $\mathbf G(\mathcal F, \Lambda)$ is a Parseval frame.
        \item There are $r$ Schwartz functions $\mathcal F:=(f_1, \dots, f_r)$ such that $\mathbf G(\mathcal F, \Lambda)$ is a frame.
        \item There are $r$ $C_Z$ functions $\mathcal F:=(f_1, \dots, f_r)$ such that $ \mathbf G(\mathcal F, \Lambda)$ is a frame.
    \end{enumerate}
\end{theorem*}

For the sake of presentation, we will discuss the material and preliminaries from the simpler file \lean{showcase_schwartz.lean}, which establishes the equivalence between the first three statements. In particular, we will avoid discussing the definition of the $C_Z$ class in Lean, as it is somewhat hard to digest. The formalized definition of the space $C_Z$ and the equivalence with the fourth statement can be found in the GitHub repository. Our theorem in Lean takes the following data.
\begin{enumerate}
    \item Two natural numbers $d,r\geq 1$ with the hypothesis that $d>1$.
    \item A lattice $L \subset \mathbb R^d$ and a lattice $\Lambda \subset \mathbb R^d\times \mathbb R^d $, with the hypothesis that $\Lambda \geq L\times L^*$.  In Lean, a lattice in $\mathbb R^d$ is defined as a $\mathbb Z$-\textit{submodule} of $\mathbb R^d$ which is topologically discrete and a $\mathbb Z$-lattice (the $\mathbb{R}$-span is the whole space).
    \item The hypothesis that $L$ is a rational lattice (all points in $\Lambda$ have rational components). 
\end{enumerate}
The theorem then states that the following three conclusions are equivalent:
\begin{enumerate}
    \item $d<r \cdot  [\Lambda : L\times L^*]$.
    \item There exist $r$ Schwartz functions such that the associated multi-window Gabor system along $\Lambda$ is a frame.
    \item There exist $r$ Schwartz functions such that the associated multi-window Gabor system along $\Lambda$ is a Parseval frame.
\end{enumerate}

Once the right definitions are in place, the translation in Lean is literal.

\begin{leancode}
theorem schwartz_gabor_frame_tfae {d r : ℕ} (hd : 1 < d)
    (L : Submodule ℤ (ℝ^[d])) [DiscreteTopology L] [IsZLattice ℝ L]
    (Λ : Submodule ℤ (ℝ^[d] × ℝ^[d])) [DiscreteTopology Λ] [IsZLattice ℝ Λ]
    (hL_rational : isRational L) (hΛ_larger: Λ ≥ L.prod (Lᵛ)):
    TFAE
      [d < r*(LatticeIndex hΛ_larger),
       ∃ windows : (SchwartzMap (ℝ^[d]) ℂ)^[r], 
        IsFrame (GaborFamily windows (Λ)),
       ∃ windows : (SchwartzMap (ℝ^[d]) ℂ)^[r],
          IsParsevalFrame 
          (GaborFamily windows (Λ))]
    := by  -- the proof starts here and spans multiple files
\end{leancode}

\subsection{Preliminary definitions}

To state \lean{schwartz_gabor_frame_tfae} we have to provide several elementary definitions that are not in Mathlib. We explain the main ones below. 

The simplest one is \lean{isRational}, which states that every element of $L$ is in $\mathbb Q^d$. The definition in Lean is a bit more verbose, as \lean{isRational} takes a $\mathbb Z$-submodule \lean{L} of $\mathbb{R}^d$ and returns the proposition ``for every element $x$ of \lean{L} there is a rational element $q$ which, when \textit{coerced} into $\mathbb{R}^d$,\footnote{in Lean, rational numbers (fractions) are not real numbers (equivalence classes of Cauchy sequences) and one must coerce them into reals. Lean is able to perform these coercions automatically most of the time, but the ``:" notation such as \lean{(q : ℝ^[d])} forces it. } is equal to $x$. In Lean, we write

\begin{leancode}
def isRational {d : ℕ} (L: Submodule ℤ (ℝ^[d])) : Prop := 
  ∀ x : L, ∃ q : ℚ^[d], x = (q : ℝ^[d])
\end{leancode}

The next necessary definition is the index of a lattice with respect to another lattice. Lattices are defined as $\mathbb Z$-modules in Mathlib, and Mathlib has a definition of index (\lean{AddSubgroup.relIndex}) for additive subgroups. In order to use it, we have to turn the lattices into additive subgroups.

\begin{leancode}
def LatticeIndex {d : ℕ}
    {Λ₀ Λ : Submodule ℤ (ℝ^[d] × ℝ^[d])} (_hBase : Λ₀ ≤ Λ) : ℕ :=
  AddSubgroup.relIndex (Λ₀.toAddSubgroup) (Λ.toAddSubgroup)
\end{leancode}

The next step is to define the Gabor family of translations and modulations. We begin by defining the phase function.
\begin{leancode}
def phase {d : ℕ} (ω : ℝ^[d]) : ℝ^[d] → ℂ :=
  fun t => Complex.exp (2 * Real.pi * Complex.I * (ω ⬝ᵥ t))
\end{leancode}

We wish to use the phase function to define the time-frequency shift $\tau_{x_0,w} f(x) := e^{2\pi i w\cdot x }f(x -x_0)$. This is slightly delicate in Lean because Schwartz functions are not just maps $\mathbb R^d \to \mathbb C$, but objects that carry much more structure (decay, differentiability, etc.). A slightly cumbersome consequence of this is that one must use \lean{SchwartzMap.smulLeftCLM}, which defines the product of a Schwartz function and a function of tame growth as a Schwartz function. Similarly, one must use \lean{SchwartzMap.compSubConstCLM}, which defines the translate of a Schwartz function as a Schwartz function. With this in mind, \lean{schwartzTimeFrequencyShift} is a function that takes $(x,\omega) \in \mathbb R^d \times \mathbb R^d$ and returns $\tau_{x,\omega}$, a map from $\mathcal{S}(\Rd)$ to itself.

\begin{leancode}
def schwartzTimeFrequencyShift {d : ℕ}:
    ℝ^[d]× ℝ^[d] →  (SchwartzMap (ℝ^[d]) ℂ → SchwartzMap (ℝ^[d]) ℂ) :=
    fun ⟨x, ω⟩ ↦ SchwartzMap.smulLeftCLM ℂ (phase ω) ∘ 
                 SchwartzMap.compSubConstCLM ℂ x
\end{leancode}

 Given a lattice in $\mathbb R^d \times \mathbb R^d$ and a family of $r$ windows, the above definition can be repackaged to define the Gabor family. Note that families (indexed sets) are stored in Lean as the maps from the index set to the family. The \lean{.toLp 2} end in the expression below means that the Schwartz function should be thought of as an $L^2$ function.
\begin{leancode}
def GaborFamily {d r : ℕ} 
    (windows : (SchwartzMap (ℝ^[d]) ℂ)^[r])
    (Λ : Submodule ℤ (ℝ^[d] × ℝ^[d]))  :
    (Fin r) × Λ  →  L²(ℝ^[d]) :=
  fun ⟨j, τ⟩ => (schwartzTimeFrequencyShift τ (windows j)).toLp 2
\end{leancode}

The last remaining definition that was not already in Mathlib is that of frames. We define frames and Parseval frames (frames with constant 1) below. They can be defined for general commutative groups \lean{H} which are inner product spaces and for an arbitrary index set\lean{ ι}. We then apply them in the particular case of $L^2$ for a Gabor family.

\begin{leancode}
def IsFrame {ι H : Type*} [NormedAddCommGroup H]  [InnerProductSpace ℂ H] 
    (atoms : ι → H) : Prop :=
  ∃ A > 0, ∃ B > 0, ∀ f : H,
      A * ‖f‖ ^ 2 ≤ ∑' i, ‖⟪f, atoms i⟫_ℂ‖ ^ 2 ∧
      ∑' i, ‖⟪f, atoms i⟫_ℂ‖ ^ 2 ≤ B * ‖f‖ ^ 2

def IsParsevalFrame {ι H : Type*} [NormedAddCommGroup H] [InnerProductSpace ℂ H]
    (atoms : ι → H) : Prop :=
  ∀ f : H, (∑' i, ‖⟪f, atoms i⟫_ℂ‖ ^ 2) = ‖f‖ ^ 2
\end{leancode}

\subsection{Convenience notation}
Certain notation, such as \lean{ℝ^[d]} to denote $\mathbb R^d$ or\lean{ Lᵛ} to denote the dual lattice, are not standard in Lean (which uses much more verbose notation). The first lines of \lean{showcase_schwartz.lean} record our notation and provide a brief explanation of our choices. Our intention when introducing this notation was to make the statements easier for non-experts in Lean to parse.

\section*{Acknowledgments}

L.L.~is grateful to the Azrieli Foundation for the award of an Azrieli Fellowship and acknowledges the support of this research by ISF Grant No.~854/25.

\bibliographystyle{plain}
\bibliography{bibfile}

\end{document}